\documentclass{article}
\usepackage[%
journal=XXX,    
lang=british,   
]{ems-journal}

\theoremstyle{definition}
\newtheorem{definition}{Definition}[section]
\newtheorem{theorem}{Theorem}[section]
\newtheorem{lemma}{Lemma}[section]
\newtheorem{proposition}{Proposition}[section]
\newtheorem{remark}{Remark}[section]
\allowdisplaybreaks

\numberwithin{equation}{section}

\begin{document}

\title{Existence of weak solutions to two-phase Stefan problems for parabolic partial differential equation system}
\titlemark{Existence of weak solutions to the two-phase Stefan problem}


%

\emsauthor{1}{
	\givenname{Hana}
	\surname{Kakiuchi}
	\mrid{1717041}
	\zblid{kakiuchi.hana}
	\orcid{0009-0002-2924-4116}}{H. Kakiuchi}

\Emsaffil{1}{
	\department{Division of Mathematical and Physical Sciences}
	\organisation{Japan Women's University}
	\rorid{04gpcyk21}
	\address{2-8-1, Mejirodai, Bunkyo-ku}
	\zip{112-8681}
	\city{Tokyo}
	\country{Japan}
	\affemail{m1816022kh@ug.jwu.ac.jp}}

\classification[35R35]{35K55}

\keywords{Free boundary problem, Stefan problem, weak solution, bread baking process}

\begin{abstract}
In our previous results, we have proposed a new free boundary problem representing the bread baking process in a hot oven. The problem consists of conservation laws for the internal energy and water mass, the Stefan condition, the boundary and initial conditions. We note that the boundary condition for   the water content contains the temperature at the boundary. Accordingly,  we do not expect a strong solution for $W^{1,2}$-initial data. To address this difficulty, we established strong solvability by assuming stricter conditions for the initial functions in \cite{AK2}. In contrast, this paper aims to prove the existence of a weak solution, when these assumptions for initial functions are relaxed.
\end{abstract}

\maketitle


\section{Introduction} 
We consider the following free boundary problem as a mathematical model for the bread baking process which we derived 
based on \cite{PS} (see \cite{AK, AK2} for details). In the problem we find a triplet $(u, w, e)$ of the temperature field $u$ and the water content $w$ on $Q(T) := (0,T) \times (0,1)$, the evaporation front $x=e(t) \mbox{ for } 0 < t < T$, where $T > 0$ is a given time, $t$ is time variable and the bread occupies the one dimensional region $(0,1)$.

Here, we note that $u$ is defined by $u = \theta - \theta_c$, where $\theta$ is the temperature and $\theta_c$ is phase transition temperature, and the common domain of $u$ and $w$ consists of the crumb region $Q_l(T, e) =  \{ (t,x) \in Q(T)| 0 < x < e(t) \mbox{ for } t \in (0,T)\}$ and the crust region $Q_a(T, e) =  \{ (t,x) \in Q(T)|  e(t) < x < 1 \mbox{ for } t \in (0,T)\}$ (see \cite{AK} for details). Moreover, the triplet $(u, w, e)$ satisfies:
 \begin{align}
& c_l \frac{\partial u}{\partial t} = k_l \frac{\partial^2 u}{\partial x^2} \quad  \mbox{ in } Q_l(T, e), \quad c_a \frac{\partial u}{\partial t} = k_a \frac{\partial^2 u}{\partial x^2} \quad \mbox{ in } Q_a(T, e),  \label{EQl}  \\
& \frac{\partial w}{\partial t} = d_l \frac{\partial^2 w}{\partial x^2}  \quad \mbox{ in } Q_l(T, e),  \quad \frac{\partial w}{\partial t} = d_a \frac{\partial^2 w}{\partial x^2}  \quad  \mbox{ in } Q_a(T, e), \label{EQa} \\
& \frac{\partial u}{\partial x}(t,0) = 0, \quad  u(t,e(t)) = 0  \quad \mbox{ for } 0 < t < T,  \label{BCl} \\
& - k_a \frac{\partial u}{\partial x}(t,1) = h( u(t,1) + \theta_c - u_b(t)) + \sigma( (u(t,1) + \theta_c)^4 - u_b(t)^4)\notag\\
&\hspace{80mm} \mbox{ for } 0 < t < T, \label{BCau} \\
& - d_a \frac{\partial w}{\partial x}(t,1) = b_1p(u(t,1) + \theta_c) -  b_2 p(u_b(t))  \quad \mbox{ for } 0 < t < T,  \label{BCaw} \\
& l w(t,e(t))  e'(t) = k_l \frac{\partial u}{\partial x}(t, e(t)-) -   k_a \frac{\partial u}{\partial x}(t, e(t)+)   \quad \mbox{ for } 0 < t < T,  \label{FBP1} \\
&   \frac{\partial w}{\partial x}(t, 0) = 0,  \quad  d_l \frac{\partial w}{\partial x}(t, e(t)-) = d_a \frac{\partial w}{\partial x}(t, e(t)+) \quad \mbox{ for } 0 < t < T,  \label{FBP2} \\
& w(t,e(t)-) = w(t,e(t)+) \quad \mbox{ for } 0 < t < T,  \label{FBP3} \\
&e(0) = e_0, u(0) = u_0, w(0) = w_0\quad \mbox{ on } [0,1],\label{IC}
\end{align} 
where $c_l$ and $c_a$ are the specific heats, $k_l$ and $k_a$ are the thermal conductivities, and $d_l$ and $d_a$ are the water diffusion coefficients in the crumb and the crust regions, respectively. Also, $h$ is the heat transfer constant and $\sigma$ is the Stefan-Boltzman constant, $u_b$ is temperature of the hot air in the oven and a given function on $[0,T]$,  $b_1$ and $b_2$ are positive constants,  $p$ is a monotone increasing and continuous function on $\mathbb{R}$, and $l$ is the latent heat. Moreover, $e_0$ is the initial position of the free boundary, $u_0$ is initial temperature field, and $w_0$ is initial  water content. Throughout this paper, the system \eqref{EQl}-\eqref{IC} is denoted by P $=$ P$(u_0, w_0)$.
\par Details of the physical background are scribed in \cite{AK}. Here, we explain the derivation of the system \eqref{EQl}-\eqref{IC}, briefly. Equation \eqref{EQl} is the heat equation and equation \eqref{EQa} is the diffusion equation for water. By symmetry of the bread, we impose the homogeneous Neumann boundary conditions \eqref{BCl} and \eqref{FBP2} at $x=0$ and the boundary conditions \eqref{BCau} and \eqref{BCaw} at $x=1$ are assumed in \cite{PS}. In this paper,  we call the equation \eqref{FBP1} the generalized Stefan condition, since the water content $w$ appears as the coefficient of the time derivative of the free boundary. Moreover, \eqref{FBP2} is  a type of  the transmission condition, and  \eqref{FBP3} means that $w(t)$ is continuous at the free boundary for each $t \in [0,T]$.
\par For the standard Stefan problem, there exists a strong solution $(u,e)$ under the condition $u_0 \in W^{1,2}(0,1)$ with some sign properties. Here, we note that the boundary condition \eqref{BCaw} contains $u(\cdot,1)$ at $x=1$, and for existence of the strong solution $w$, it is necessary that the time derivative $u_t(\cdot,1)$ is integrable. However, even if $u$ is the strong solution, the integrability of $u_t(\cdot,1)$ does not guaranteed in usual. For this difficulty, by the assumption  $u_0 \in W^{1,2}(0,1) \cap W^{2,2}(e_0,1)$ and $w_0 \in W^{1,2}(0,1)$, we obtained the integrability and proved the strong solvability in \cite{AK2}. On the other hand, it is quite interesting to discuss solvability under the assumption $u_0 \in W^{1,2}(0,1)$. Thus, the aim of this paper is to establish existence of weak solutions in case $u_0 \in W^{1,2}(0,1)$ and $w_0 \in L^2(0,1)$. We note that we can relax the regularity assumption on $u_0$ and $w_0$ by imposing suitable conditions that allow us to apply the comparison principle (see Remark \ref{rem1} in detail). The uniqueness of weak solutions to our free boundary problem is not easy, because of the generalized Stefan condition and lack of regularity for $u$, $w$ and $e$.
\section{Main result}
In this section, we give a definition for a weak solution of P and a theorem concerned with  the existence of the solution.
\begin{definition}\label{def1}
For given time $T$ a triplet $(u, w, e)$ of functions $u$, $w$ and $e$ is a solution of P on $[0,T_0]$ for  $0 < T_0 \leq T$ if $(u, w, e)$  satisfies (S1)-(S4).\\
(S1) $e \in W^{1, 4}(0, T_0)$ and $0 < e < 1  \mbox{ on }  [0, T_0]$.\\
(S2) $u \in W^{1,2}(0,T_0;L^2(0, 1)) \cap L^{\infty} (0, T_0; W^{1,2}(0, 1))$, $u_{xx}  \in L^2(Q_l(T_0, e))$, $L^2(Q_a(T_0, e))$.\\
 (S3) $w \in L^{\infty}(0,T_0;L^2(0, 1)) \cap L^{2} (0, T_0; W^{1,2}(0, 1))$.\\
 (S4) \eqref{EQl}, \eqref{BCl}, \eqref{BCau}, \eqref{FBP1}, $e(0)=e_0$, $u(0)=u_0$ on $[0,1]$ and the following weak formulation for $w$ hold:
  \begin{align}
& - \int^{T_0}_0  \int^1_0 w \eta_t dx dt + d_l \int^{T_0}_0  \int^e_0 w_x \eta_x dx dt +  d_a \int^{T_0}_0  \int^1_e w_x \eta_x dx dt  \notag \\ 
& =    \int^1_0 w_0(x) \eta(0,x) dx - \int^{T_0}_0 \{ b_1 p (u(t,1) + \theta_c) - b_2 p(u_b(t))\} \eta(t,1) dt \label{weak}
 \end{align} 
 \hspace{15mm} for $\eta \in W^{1,2}(0, T_0; L^2(0, 1)) \cap L^{2}(0, T_0; W^{1,2}(0,1))$ and $\eta(T_0)=0$.
\end{definition}

\begin{theorem}\label{th1}
Assume\\
(A1) $p \in C^1({\mathbb R})$, $0 \leq p, p' \leq M_p$ on ${\mathbb R}$,  where $M_p$ is a positive constant.\\
(A2)   $0 <  e_0 < 1$.\\
(A3) $u_0 \in W^{1,2}(0, 1)$, $u_0 \geq 0$ on $[e_0,1]$, $u_0 \leq 0$ on  $[0,e_0]$.\\
(A4) $w_0  \in L^2(0,1)$, $w_0 \geq \delta_1$ a.e. on $[0,1]$, where $\delta_1$ is a positive constant.\\
(A5) $u_b$ is a positive constant with $u_b \geq u_0 +  \theta_c \mbox{ on } [0,T]$ and $b_1 \leq b_2$.\\
Then,  P has at least one solution $(u, w, e)$ on $[0,T_0]$ for some $T_0 \in (0,T]$. 
\end{theorem}
\begin{remark}\label{rem1}
We contrast the assumptions on the initial and boundary functions imposed here with those in our previous results \cite{AK} and \cite{AK2}. In \cite{AK}, we established the existence and uniqueness of strong solutions to the problem with the approximated boundary condition, and in \cite{AK2}, under (A1), (A2), (A3), (A4)', (A5)' and (A6)', the strong solvability has been obtained as \cite[Theorem 2.1]{AK2}, where (A4)' - (A6)' are:\\
(A4)' $w_0  \in W^{1,2}(0,1)$, $w_0(e_0) > 0$.\\
(A5)' $u_b  \in W^{1,2}(0,T)$, $u_b \geq  \theta_c \mbox{ on } [0,T]$.\\
(A6)' $u_0 \in W^{2,2}(e_0,1)$, $u_{0x} = 0$, $-k_a u_{0x}(1) =  h( u_{0 }(1) + \theta_c - u_b(0)) + \sigma( (u_{0 }(1) + \theta_c)^4 - u_b(0)^4)$.\\
To compensate for loss of regularity in $w$ by omitting (A6)', we impose (A4) and (A5), which provide the desired lower estimate for $w$. Based on this lower estimate, we can obtain uniform estimates for the free boundary $e$ with \eqref{FBP1}. The energy inequality \eqref{en3} is often employed in the analysis of two-phase Stefan problem. However, because $w$ is defined as a weak solution, estimate \eqref{en3} alone is insufficient to establish the convergence of the approximate solutions (see Section \ref{proof-prop1}). The key to proving the convergence lies in improving the regularity of the free boundary from $W^{1,3}$ to $W^{1,4}$ by applying the Gagliardo-Nirenberg inequality.
\end{remark}
\par Throughout this paper,  for simplicity we put  $g:  {\mathbb R} \to {\mathbb R}$ as follows:
\begin{eqnarray*}
g(r) = h (r + \theta_c - u_b) + \sigma ((r + \theta_c)^4 - u_b^4) \quad \mbox{for } t \in [0,T] \mbox{ and } r \in \mathbb{R}.
\end{eqnarray*}
To apply our previous result, we construct an approximate sequence for the initial function $u_0$ such that  assumptions of Theorem 2.1 in \cite{AK2} hold.
\begin{lemma}\label{lem1}
Assume (A3) - (A5). Then, there exist  approximate sequences $\{u_{0n}\}$ and $\{w_{0n}\}$ to $u_0$ and $w_0$, respectively, such that
\begin{align*}
& \{u_{0n}\} \subset W^{1,\infty}(e_0, 1) \cap W^{2,2}(x_0, 1) \quad \mbox{for some } x_0 \in (e_0,1),\\
&  -k_a u_{0nx}(1) = g(0,u_{0n}(1))\\
& u_{0n} \geq 0 \quad \mbox{on } [e_0,1], \quad u_{0n} \leq 0 \quad \mbox{on } [0, e_0],\\
&u_b \geq \theta_c + u_{0n} \quad \mbox{on } [0,1] \quad \mbox{for large }n,\\
& \{w_{0n}\} \subset W^{1,2}(0, 1), \quad w_{0n} \geq \delta_1 \quad \mbox{on } [0, 1] \mbox{ for } n \geq1,\\
&u_{0n} \to u_0 \quad \mbox{in } W^{1,2}(0,1) \mbox{ and } w_{0n} \to w_0 \quad \mbox{in } L^2(0,1)  \mbox{ as } n  \to \infty.
\end{align*}
  \end{lemma}
\begin{proof} 
Let  $x_0 \in (e_0,1)$, and put 
\begin{align*}
\overline{u}_0(x) = 
\begin{cases}
&2u_0(x_0) - u_0(2 x_0 - x)\quad \mbox{for } e_0 \leq x < x_0, \\
&u_0(x) \quad  \mbox{for } x_0 \leq x \leq 1,\\
&u_0(2-x) \quad \mbox{for } 1 < x < 2.
\end{cases}
\end{align*}
It is clear that $\overline{u}_0 \in W^{1,\infty}(e_0,1)$, $J_{\varepsilon} \ast \overline{u}_0 \in C^{\infty}(\mathbb{R})$ for $\varepsilon > 0$, where $J_{\varepsilon}$ is the standard mollifier in $\mathbb{R}$ and $``\ast"$ indicates the convolution. Here, we put
\begin{align*}
\overline{u}_{0n}(x) = 
\begin{cases}
&u_0(x) \quad  \mbox{for } 0 \leq x \leq x_0,\\
&(J_{1/n} \ast \overline{u}_0)(x) \quad \mbox{for } x_0 < x \leq 1,
\end{cases}
\end{align*}
for $n \geq n_0$, where $2/n_0 < x_0 - e_0$. Easily, we get $\overline{u}_{0n} \in W^{1,\infty}(0,1)$ for $n \geq n_0$, because of $(J_{1/n} \ast \overline{u}_0)(x_0) = u_0(x_0)$ for $n \geq n_0$. Obviously,  $\overline{u}_{0n} \geq 0$ on $[e_0,1]$, $\overline{u}_{0n} \leq 0$ on  $[0, e_0]$, $\overline{u}_{0nx}(1) = 0$ and $\overline{u}_{0n} \in W^{2,2}(x_0,1)$ for $n \geq n_0$. Here, we define a function $v_{0n}$ on $[0,1]$ by
\begin{align*}
v_{0n}(x) = 
\begin{cases}
&0 \quad \mbox{for } 0 \leq x \leq 1-\frac{2}{n}, \\
&{4n} g(\overline{u}_{0n}(1)) \left( \cos ({n \pi}(x - (1-\frac{2}{n}))) - 1\right) \quad  \mbox{for } 1-\frac{2}{n} \leq x \leq 1-\frac{1}{n},\\
&\frac{1}{2}g(\overline{u}_{0n}(1))\left( n(x - (1-\frac{1}{n}))^2 - \frac{1}{n} \right) \quad \mbox{for } 1-\frac{1}{n} < x \leq 1,
\end{cases}
\end{align*}
for $n \geq n_0$. From this definition we see that $v_{0n} \geq 0$ on $[e_0,1]$ and $v_{0n} \in W^{2,2}(0,1)$ for $n \geq n_0$. Indeed, by $\overline{u}_{0n}(1) \leq u_b - \theta_c$, we have
\begin{align*}
g(\overline{u}_{0n}(1)) &= h (\overline{u}_{0n}(1) + \theta_c - u_b) + \sigma ((\overline{u}_{0n}(1)+ \theta_c)^4 - u_b^4)\\
& \leq 0 \quad \mbox{for } n \geq n_0.
\end{align*}
Accordingly, it holds $v_{0n} \geq 0$ on $[e_0,1]$ for $n \geq n_0$. We can get $v_{0n} \in W^{2,2}(0,1)$ for $n \geq n_0$, easily. Moreover, we can show that $v_{0n} \to 0$ in $C([0,1])$ as $n \to \infty$. In fact, first it is clear that $\{ \overline{u}_{0n}(1) \}$ is bounded in $\mathbb{R}$, also $\{ g(\overline{u}_{0n}(1)) \}$ is bounded. Next, let $n \geq n_0$. Then, we see that for $1-\frac{2}{n} \leq x \leq 1-\frac{1}{n}$, $|v_{0n}(x)| \leq \frac{1}{2n} |g(\overline{u}_{0n}(1))|$, and for $1-\frac{1}{n} < x \leq 1$
\begin{align*}
|v_{0n}(x)| &\leq \frac{1}{2} |g(\overline{u}_{0n}(1))| \left| n(x - (1-\frac{1}{n}))^2 - \frac{1}{n} \right|\\
&\leq \frac{1}{2} |g(\overline{u}_{0n}(1))| \left| n(1 - (1-\frac{1}{n}))^2+ \frac{1}{n} \right|\\
&\leq \frac{1}{n} |g(\overline{u}_{0n}(1))|.
\end{align*}
Thus, $v_{0n} \to 0$ in $C([0,1])$ as $n \to \infty$ is true. Finally, we set  $u_{0n} = \overline{u}_{0n} + v_{0n}$ for $n \geq n_0$.  To complete the proof, it is sufficient to show that $u_{0n} + \theta_c \leq u_b$ on $[0,1]$ for $n \geq n_1$, where $n_1$ is some positive integer with $n_1 \geq n_0$. Here, on account of the assumption $u_{0} + \theta_c < u_b$ on $[0,1]$, we can take $m \in \mathbb{R}$ and $n_1 \geq n_0$ such that $u_{0} + \theta_c \leq m < u_b$ and $0 \leq v_{0n} \leq u_b - m$ for $n \geq n_1$. Since the existence of $\{w_{0n}\}$ satisfying the desired conditions is easily proved, this lemma is proved.
\end{proof}
The next lemma guarantees the existence of strong solutions of P$(u_{0n}, w_{0n})$ locally in time for $n$.
  \begin{lemma}\label{lem-strong}(cf. \cite[Theorem 2.2 and Lemma 3.2]{AK2})
  Assume (A1)-(A5) and for $n$ let $u_{0n}$ and $w_{0n}$ be the functions defined in Lemma \ref{lem1}. Then, there exists $T_n \in (0,T]$ such that P$(u_{0n}, w_{0n})$  has a strong solution $(u_n, w_n, e_n)$ on $[0,T_n]$. Namely, $u_n \in W^{1,2}(0,T_n;L^2(0, 1))$\\ $\cap L^{\infty} (0, T_n; W^{1,2}(0, 1))$, $u_{nxx}  \in L^2(Q_l(T_n, e_n))$, $L^2(Q_a(T_n, e_n))$, $u_{nx}(\cdot, e_n(\cdot)\pm) \in L^{\infty}(0,T_n)$, $u_n(\cdot, 1) \in W^{1,2}(0,T_n)$, $u_n \geq 0$ on $Q_a(T_n, e_n)$, $u_n \leq 0$ on $Q_l(T_n, e_n)$,  $w_n \in W^{1,2}(0,T_n;L^2(0, 1))$\\ $\cap L^{\infty} (0, T_n; W^{1,2}(0, 1))$, $w_{nxx}  \in L^2(Q_l(T_n, e_n))$, $L^2(Q_a(T_n, e_n))$, $w_n(t, e_n(t)) \geq \delta_n \mbox{ for any } t$\\$ \in [0,T_n]$, where $\delta_n$ is a positive constant, $e_n \in W^{1, \infty}(0, T_n)$, $0 < e_n < 1  \mbox{ on }  [0, T_n]$ and \eqref{EQl}-\eqref{IC} hold.
   \end{lemma}
   \begin{remark}
  In \cite[Theorem 2.1]{AK2}, the conditions $u_{0x}(0) = 0$ and $u_0 \in W^{2,2}(e_0, 1)$ are required for the strong solvability. However, from the proof of \cite[Theorem 2.1]{AK2}, we observe that $u_{0x}(0) = 0$ is not necessary and $u_0 \in W^{2,2}(d_0, 1)$ is sufficient for the existence of a strong solution, where $d_0$ is a some positive constant with $e_0 < d_0 <1$, since we need the only high regularity for $u$ near $x=1$. Hence, the conditions obtained in \ref{lem1} guarantee the existence of the  strong solutions for each $n$.
      \end{remark}
In our proof of Theorem \ref{th1}, the following proposition is very important.
\begin{proposition}\label{prop1}
Under the same assumption as in Lemma \ref{lem-strong}, there exist $T_0 \in (0,T]$ and $\delta > 0$ such that $T_n \geq T_0$ and $\delta \leq e_n \leq  1-\delta$ on $[0,T_0]$ for $n \geq 1$, where $T_n$ is a positive number obtained by Lemma \ref{lem-strong}.
  \end{proposition}
  In order to prove Proposition \ref{prop1}, we show the following Lemmas \ref{lem-a}, \ref{lem-b} and \ref{lem-c}. 
\begin{lemma}\label{lem-a}(cf. \cite[Lemma 7.1]{AK2} ) For any $n$ let $(u_n, w_n, e_n)$ be a solution of P($u_{0n}, w_{0n}$) on $[0, T_n]$. Assume (A1)-(A5). Then it holds that $u_n + \theta_c \leq u_b$ in $\overline{Q(T_n)}$.
\end{lemma}
\begin{proof}
By (A2) and the assumption $u_b \geq u_{0n} + \theta_c$ on $[0,1]$ it holds that $u_{0n}(e_0) = 0$ and $u_b \geq \theta_c$. Accordingly, for any $t \in [0, T]$, by multiplying $[u_n(t) + \theta_c -u_b]^+$ on both sides of  \eqref{EQl} and
integrating it with $x$ on $[0, e_n(t)]$, $[e_n(t), 1]$, respectively, we obtain
\begin{align*}
\int^{e_n}_0 c_l u_{nt}  [u_n + \theta_c - u_b]^+ dx &= \int^{e_n}_0 k_l u_{nxx}  [u_n + \theta_c - u_b]^+ dx\\
&=- \int^{e_n}_0 k_l u_{nx}  ([u_n + \theta_c - u_b]^+)_x dx\\
&=- \int^{e_n}_0 k_l   \left|([u_n + \theta_c - u_b]^+)_x\right|^2 dx\\
&\leq 0 \text{ a.e. on } [0,T],
\end{align*}
and from the definition  of  the function $g$ it follows 
\begin{align*}
&\int^1_{e_n} c_a u_{nt}  [u_n + \theta_c - u_b]^+ dx\\
 &= \int^1_{e_n} k_a u_{nxx}  [u_n + \theta_c - u_b]^+ dx\\
&=g( u_n(\cdot,1))  [u_n(\cdot,1) + \theta_c - u_b]^+ - \int^1_{e_n} k_a u_{nx}  ([u_n + \theta_c - u_b]^+)_x dx\\
&\leq - \int^1_{e_n} k_a   \left|([u_n + \theta_c - u_b]^+)_x\right|^2 dx\\
&\leq 0 \text{ a.e. on } [0,T].
\end{align*}
Hence, we have
\begin{align*} 
\int^1_0 u_{nt}  [u_n + \theta - u_b]^+ dx \leq 0 \text{ a.e. on } [0,T].
\end{align*}
Since $u_b$ is a constant, we see that
\begin{align*} 
\frac{1}{2} \frac{d}{dt} \int^1_0  \left| [u_n + \theta_c - u_b]^+\right|^2 dx \leq 0 \text{ a.e. on } [0,T].
\end{align*}
By the assumption $u_b \geq u_{0n} + \theta$ on $[0,1]$, it holds that
\begin{align*} 
 \int^1_0  \left| [u_n(t) + \theta_c - u_b]^+\right|^2 dx &\leq  \int^1_0  \left| [u_{0n} + \theta_c - u_b]^+\right|^2 dx\\
&=0 \text{ on } [0,T].
\end{align*}
Thus, Lemma \ref{lem-a} has been proved.
\end{proof}
\begin{lemma}\label{lem-b}(cf. \cite[Lemma 7.2]{AK2} ) For any $n$ let $(u_n, w_n, e_n)$ be a solution of P($u_{0n}, w_{0n}$) on $[0, T_n]$. Suppose the same assumption as in Lemma \ref{lem-a}, then there exists $\delta_1$ independent of $n$ such that $w_n \geq \delta_1$ on $\overline{Q(T_n)}$.
\end{lemma}
\begin{proof}
By the assumption $w_0 > 0$ on $[0,1]$,  we can take $\delta_1$ such that $w_0(x) \geq \delta_1$ for $0 \leq x \leq 1$.
For any $t \in [0, T]$ by multiplying $[-w_n(t) + \delta_1]^+$ on both sides of  \eqref{EQa} and
integrating it with $x$ on $[0, e_n(t)]$, $[e_n(t), 1]$, respectively, and integration by parts, we have
\begin{align*} 
&\int^1_0 w_{nt}  [-w_n + \delta_1]^+ dx\\
&= -\int^{e_n}_0 d_l w_{nx}  ([-w_n + \delta_1]^+)_x dx - \int^1_{e_n} d_a w_{nx}  ([-w_n + \delta_1]^+)_x dx\\
&\quad \, + d_a w_{nx}(\cdot,1)[-w_n(\cdot,1) + \delta_1]^+\\
&= \int^{e_n}_0 d_l  \left| ([-w_n + \delta_1]^+)_x\right|^2 dx + \int^1_{e_n} d_a \left|  ([-w_n + \delta_1]^+)_x\right|^2 dx\\
&\quad \, - \{b_1p(u_n(\cdot,1) + \theta_c) -  b_2 p(u_b)\}[-w_n(\cdot,1) + \delta_1]^+   \text{ a.e. on } [0,T].
\end{align*}
 Moreover, by the assumption $b_1 \leq b_2$, Lemma \ref{lem-a} and the monotonicity of $p$,  we obtain
  
\begin{align*} 
\frac{1}{2}\frac{d}{dt} \int^1_0 \left| [-w_n + \delta_1]^+\right|^2 dx &\leq  \{b_1p(u_n(\cdot,1) + \theta_c) -  b_2 p(u_b)\}[-w(\cdot,1) + \delta_1]^+\\
&\leq 0 \text{ a.e. on } [0,T].
\end{align*}

Thanks to  $w_0(x) \geq \delta_1$ for $x \in [0,1]$, it follows that
\begin{align*} 
 \int^1_0 \left| [-w_n + \delta_1]^+\right|^2 dx 
 &\leq \int^1_0 \left| [-w_{0n} + \delta_1]^+\right|^2 dx\\
  &\leq 0 \quad \ \text{on } [0,T] \mbox{ for any } n.
\end{align*}
 Thus, we can show Lemma \ref{lem-b}.
\end{proof}
\begin{lemma}\label{lem-c}(cf. \cite[Lemma 7.3]{AK2} ) For any $n$ let $(u_n, w_n, e_n)$ be a solution of P($u_{0n}, w_{0n}$) on $[0, T_n]$. Suppose the same assumption as in Lemma \ref{lem-a}, then there exists a positive constant $C_*$ independent $n$ such that 
\begin{equation}
 \int^t_0 \left|e'_n\right|^3 d\tau + \left|u_n(t)\right|_{W^{1,2}(0,1)}^2 + \int^t_0 \left|u_{nt}\right|^2_{L^2(0,1)}d\tau   \leq C_* \text{ for } 0 \leq t \leq T_n.\label{en3}
\end{equation}
\end{lemma}
\begin{proof}
 By multiplying $k_l u_t$, $k_a u_t$ on both sides of the first equation of \eqref{EQl}, the second one and integrating it with $x$ on $[0, e_n]$, $[e_n, 1]$, respectively, it holds that
\begin{align*} 
& \delta_* \int^1_0  \left|u_{nt}\right|^2 dx + \frac{k_l^2}{2} \frac{d}{dt} \int^{e_n}_0\left|u_{nx}\right|^2 dx   + \frac{k_a^2}{2} \frac{d}{dt} \int^1_{e_n}\left|u_{nx}\right|^2 dx \\
&\quad \, + \frac{k_l^2}{2} u_{nx}(\cdot, e_n-)^2 e'_n - \frac{k_a^2}{2} u_{nx}(\cdot, e_n+)^2 e'_n\\
&\quad \,  + k_a  (h( u_n(\cdot,1) + \theta_c - u_b) + \sigma( (u_n(\cdot,1) + \theta_c)^4 - u_b^4))u_{nt}(\cdot, 1)\\
&\leq 0 \text{ a.e. on } [0,T_n],
\end{align*}
where $\delta_* = \min \{ c_l k_l, c_a k_a \}$. Since we approximate $u_0$, $u_{nt}(\cdot,1)$ is well defined(see \cite[Proposition 2.7]{AK2}).
Easily, we have
\begin{align*} 
 \frac{k_l^2}{2} u_{nx}(\cdot, e_n-)^2 e'_n - \frac{k_a^2}{2} u_{nx}(\cdot, e_n+)^2 e'_n &= \frac{\left|e'_n\right|^2}{2}(k_l  u_{nx}(\cdot, e_n-) + k_a  u_{nx}(\cdot, e_n+) )l w_n(\cdot, e_n)\notag\\
 &=: I \text{ a.e. on } [0,T_n].
\end{align*}
If $e'_n(t) > 0$ for some $t \in [0,T_n]$, then $k_l  u_{nx}(t, e_n(t)) > k_a  u_{nx}(t, e_n(t))$ implies that \[\left|e'_n(t)\right| = k_l  u_{nx}(t, e_n(t)-) - k_a  u_{nx}(t, e_n(t)+).\] Thus, we obtain
\begin{samepage}
  \begin{align*}
I(t) &\geq  \frac{\left|e'_n(t)\right|^2}{2}(k_l  u_{nx}(t, e_n(t)-) + k_a  u_{nx}(t, e_n(t)+) )l w_n(t, e_n(t))\\
&\geq  \frac{\left|e'_n(t)\right|^2}{2}(k_l  u_{nx}(t, e_n(t)-) - k_a  u_{nx}(t, e_n(t)+) )l w_n(t, e_n(t))\\
&= \frac{\left|e'_n(t)\right|^3}{2}l^2 w_n(t, e_n(t))^2 \text{ a.e. on } [0,T_n].
\end{align*}
\end{samepage}

We note that $u_{nx}(t,e_n(t)) \geq 0$ because of $u_n(t,e_n(t)) = 0$ and $u_n(t,x) \geq 0$ for all $x \in [e_n(t),1]$.
In case $e'_n(t) \leq 0$  for some $t \in [0,T_n]$, we can get the same inequality. Also, we have
\begin{align*}
& (h( u_n(\cdot,1) + \theta_c - u_b) + \sigma( (u_n(\cdot,1) + \theta_c)^4 - u_b^4))u_{nt}(\cdot, 1)\\
&=\frac{d}{dt}G_n + hu'_bu_n(\cdot, 1) + 4\sigma u_b^3 u'_b u_n(\cdot, 1) \text{ a.e. on } [0,T_n],
\end{align*}
where
\begin{align*}
G_n &= \frac{h}{2}(u_n(\cdot,1) + \theta_c)^2 - h u_b u_n(\cdot,1) +  \frac{\sigma}{5}(u_n(\cdot,1) + \theta_c)^5 - \sigma u_b^4 u_n(\cdot,1) \text{ on } [0,T_n].
\end{align*}
Here, we can get $u_n(\cdot,1) \geq 0$ on $[0,T_n]$ in the similar way of Lemma 3.2 in \cite{AK2}. Therefore, by applying Young's inequality, we obtain
\begin{align*}
G_n &=  \frac{h}{2}(u_n(\cdot,1) + \theta_c)^2 - h u_b (u_n(\cdot,1) + \theta_c) + h u_b  \theta_c \\
&\quad \,  +  \frac{\sigma}{5}(u_n(\cdot,1) + \theta_c)^5 - \sigma u_b^4 (u_n(\cdot,1) + \theta_c) +  \sigma u_b^4\theta_c\\
&\geq  \frac{h}{2}(u_n(\cdot,1) + \theta_c)^2 - h u_b (u_n(\cdot,1) + \theta_c) - \sigma u_b^4 (u_n(\cdot,1) + \theta_c)\\
&\geq  \frac{h}{4}(u_n(\cdot,1) + \theta_c)^2 - 2h u_b^2 - \frac{2}{h}\sigma^2 u_b^8\\
&\geq -C_b \text{ a.e. on } [0,T_n],
\end{align*}
where $\displaystyle C_b = 2h u_b^2 - \frac{2}{h}\sigma^2 u_b^8$. 
By putting
\begin{align*} 
E_n = \frac{k_l^2}{2} \int^{e_n}_0\left|u_{nx}\right|^2 dx   + \frac{k_a^2}{2} \int^1_{e_n}\left|u_{nx}\right|^2 dx + k_a(G_n + C_b) \text{ on } [0,T_n],
\end{align*}
we see that
\begin{align*} 
& \delta_* \int^1_0  \left|u_{nt}\right|^2 dx +  \frac{\left|e'_n\right|^3}{2}l^2 w_n(\cdot, e_n)^2+  \frac{d}{dt}E_n \leq 0 \text{ a.e. on } [0,T_n].
\end{align*}
Thus, we obtain
\begin{align*} 
 \delta_* \int^t_0 \int^1_0  \left|u_{nt}\right|^2 dxd\tau +  \frac{l^2}{2}\int^t_0 \left|e'_n\right|^3 w_n(\cdot, e_n)^2 d\tau + E_n(t) \leq E_n(0) \text{ for } t \in [0,T_n].
\end{align*}
Hence,  on account of Lemma \ref{lem-b}, Lemma \ref{lem-c} has been proved.
\end{proof}
  \begin{proof}[Proof of Proposition \ref{prop1}]
  Let $(u_n, w_n, e_n)$ be the solution of P$(u_{0n}, w_{0n})$  on $[0, T_n]$ for $n \geq 1$, and take $\delta_0 > 0$ such that $0 < 2 \delta_0 \leq e_0 < x_0 < 1 - 2 \delta_0 < 1$ and $[0, \tilde{T}_n)$ be the maximal interval of existence of the solution to  P$(u_{0n}, w_{0n})$. Moreover, we choose $T_0 \in (0,T]$ with $C_*^{1/3} T_0^{2/3} < \delta_0$, and put $\hat{T}_n = \min \{ T_0, \tilde{T}_n \}$, where $C_*$ is the positive constant given in Lemma \ref{lem-c}. From now on we establish the existence a solution of P$(u_{0n}, w_{0n})$ on $[0, T_0]$ for any $n$. If $\tilde{T}_n \geq T_0$, then it is true. Otherwise, it is sufficient to show that the solution can be extend beyond $\tilde{T}_n$. First, thanks to Lemma \ref{lem-c}, it holds that
\begin{align} 
|u_n(t) - u_n(t')|_{L^2(0,1)} &\leq \int^t_{t'} |u_{n \tau}|_{L^2(0,1)} d\tau \notag \\
&\leq C_1(t-t')^{1/2} \quad \mbox{for } 0 < t' < t < T_n,\label{Aomo}
\end{align}
where $C_1$ is a positive constant. This yields that $\{ u_n(t) \}_{t \uparrow T_n}$ is a Cauchy sequence in $L^2(0,1)$. By Lemma \ref{lem-c}, we see that
\begin{align}
|e_n(t) - e_n(t')| &\leq \int^t_{t'}|e'_n|d\tau \notag\\
&\leq C_2 (t-t')^{\frac{2}{3}} \quad \mbox{for } 0 < t' < t < \hat{T}_n, \label{cauchy}
\end{align}
and
\begin{align*}
|e_n(t) - e_0| &\leq \int^t_{0}|e'_n|d\tau \\
&\leq C_*^{1/3} T_0^{2/3} \\
&< \delta_0 \quad \mbox{for } 0 <  t < \hat{T}_n,
\end{align*}
where $C_2$ is a positive constant. It is clear that $\delta_0 \leq e_n \leq 1 - \delta_0$ on $[0,\hat{T}_n)$. On account of \eqref{cauchy}, $\{ e_n(t) \}_{t \uparrow \hat{T}_n}$ is also a Cauchy sequence in $\mathbb{R}$. Thus, there exist $e_{*} \in \mathbb{R}$ and $u_{*} \in L^2(0,1)$ such that
\begin{align*} 
e_n(t) \to e_{*} \quad \mbox{in } \mathbb{R} \mbox{ as } t \uparrow \hat{T}_n, \quad \delta_0 \leq e_{*} \leq 1-\delta_0,
\end{align*}
and
\begin{align} 
u_n(t) \to u_{*} \quad \mbox{in } L^2(0,1) \mbox{ as } t \uparrow \hat{T}_n.\label{ri}
\end{align}
Moreover, thanks to Lemmas \ref{lem-c}, we observe that  $\{ u_n(t)| 0 \leq t < \hat{T}_n \}$ is bounded in $W^{1,2}(0,1)$. Since $u_{0n} \in W^{2,2}(x_0,1)$, we can obtain the following estimate in the similar way to \cite[Lemma 4.4]{AK2}: $$|u_{nt}(t)|_{L^2(1-\delta_0,1)} \leq C_n \mbox{ for } 0 \leq t \leq \hat{T}_n,$$ where $C_n$ is a positive constant. Hence, $u_{nxx} = \frac{c_a}{k_a}u_{nt}$ in $(0, \hat{T}_n) \times (1-\delta_0, 1)$ implies that $\{ u_n(t)| 0 \leq t < \hat{T}_n \}$ is also bounded in $W^{2,2}(1-\delta_0,1)$. Immediately, it holds that
\begin{align*} 
u_n(t) \to u_{*} \quad \mbox{in } C([0,1]),  \mbox{ and weakly in } W^{1,2}(0,1) \mbox{ and } W^{2,2}(1-\delta_0,1) \mbox{ as } t \uparrow \hat{T}_n.
\end{align*}
\par Furthermore, by Lemma \ref{lem-strong}, we see that $u_n(t) \geq 0$ on $[e_n(t),1]$ for $t \in [0,\hat{T}_n)$. Hence, it is easy to obtain $u_{*} \geq 0$ on $[e_{*}, 1]$. Similarly, we can prove $u_{*} \leq 0$ on $[0, e_{*}]$.
\par Next, based on \eqref{ri}, \cite[Lemmas 4.3 and 4.6]{AK} and \cite[Lemma 5.2]{AK2} guarantee that
$$\int^1_0 |w_{nx}(t)|^2 dx + \int^t_0 \int^1_0 |w_{n \tau}(\tau)|^2 dx d\tau \leq C_n^{(1)} \mbox{ for } 0 \leq t \leq T_n,$$
where $C_n^{(1)}$ is a positive constant which may depend on $n$. Similarly to \eqref{Aomo},  $\{ w_n(t) \}_{t \uparrow \hat{T}_n}$ is a Cauchy sequence in  $L^2(0,1)$. Accordingly, there exists $w_{*} \in L^2(0,1)$ such that
\begin{align*} 
w_n(t) \to w_{*} \quad \mbox{in } L^2(0,1) \mbox{ as } t \uparrow \hat{T}_n.
\end{align*}
Moreover, the boundedness of  $\{ w_n(t)| 0 \leq t < \hat{T}_n \}$ in $W^{1,2}(0,1)$ implies that
\begin{align*} 
w_n(t) \to w_{*} \quad \mbox{in } C([0,1]) \mbox{ and weakly in } W^{1,2}(0,1) \mbox{ as } t \uparrow \hat{T}_n.
\end{align*}
 Because of $w_n(t) \geq \delta_1$ on $\overline{Q(\hat{T}_n)}$ by Lemma \ref{lem-b}, we have $w_{*} \geq \delta_1$ on $[0,1]$.
\par Finally, we show $-k_a u_{*  x}(1) = g( u_{*}(1))$. By Lemma \ref{lem4}, $\{ u_{nx}(t)| 0 \leq t < \hat{T}_n \}$ is bounded in $W^{1,2}(1-\delta_0,1)$, $\{ u_{nx}(t)| 0 \leq t < \hat{T}_n \}$ is bounded in $W^{1,2}(1-\delta_0,1)$. Hence, 
\begin{align*} 
u_{nx}(t) \to u_{*x} \quad \mbox{in } C([1-\delta_0,1]) \mbox{ as } t \uparrow \hat{T}_n.
\end{align*}
It is also clear that
\begin{align} 
u_n(t) \to u_{*} \quad \mbox{in } C([1-\delta_0,1]) \mbox{ as } t \uparrow \hat{T}_n.\label{con8}
\end{align}
By using  \eqref{con8}, we obtain $-k_a u_{*  x}(1) = g( u_{*}(1))$. Hence, thanks to Lemma \ref{lem-strong}, we see that P$(u_{0n}, w_{0n})$ has a solution on $[0, \tilde{T}_n]$ for some $\tilde{T}_n > \hat{T}_n$. This guarantees that P$(u_{0n}, w_{0n})$ has a solution 
$(u_n, w_n, e_n)$ on $[0,T_0]$. Obviously, we have $\delta \leq e_n \leq 1-\delta$ on $[0,T_0]$ for $n \geq 1$.
   \end{proof}
\section{Proof of Theorem \ref{th1}}\label{proof-prop1}
The aim of this section is to prove Theorem \ref{th1}. Before presenting its proof, we establish some uniform estimates for $e_n$ and $w_n$. Throughout this section, we assume (A1)-(A5) and let $T_0$ be the positive constant obtained in Proposition \ref{prop1}.
\begin{lemma}\label{lem4}(cf.  \cite[Lemma 5.2]{AK2}) 
Let $(u_n, w_n, e_n)$ be a solution of P$(u_{0n}, w_{0n})$ for $n \geq 1$. Then there exists a positive constant $C_3 = C_3(\delta, M, T_0, |u_b|_{W^{1,2}(0,T)}, |u_{0n}|_{W^{1,2}(0,1)}, |w_{0n}|_{L^2}$\\$_{(0,1)})$  such that
\begin{align}
 \int^1_0 |w_n(t,x)|^2dx +\int^{t}_{0} \!\!\! \int^{1}_{0} |w_{nx}(\tau,x)|^2 dx d\tau \leq C_3 \quad \mbox{for } 0 \leq t \leq T_0.\label{wwx}
\end{align}
\end{lemma}
\begin{proof}
For any $t \in [0,T]$, by multiplying $w_n$ on both sides of \eqref{EQa} and integrating it with $x$ on $[0, e(t)]$ and $[e(t),1]$, we have
\begin{align*}
&\frac{1}{2} \frac{d}{dt} \int^1_0 |w_n|^2dx + d_l \int^{e}_0 |w_{nx}|^2 dx  + d_a \int_{e}^1 |w_{nx}|^2 dx\\
&= d_a w_{nx}(\cdot,1) w_n(\cdot,1)\\
&= - (b_1 p(u_n(\cdot,1) + \theta_c) - b_2 p(u_b))w_n(\cdot,1)\\
 &\leq \frac{2}{d_a} |b_1 M_p + b_2 M_p|^2 + \frac{d_a}{8}|w_n(\cdot,1)|^2\\
    &\leq \frac{2}{d_a} |b_1 M_p + b_2 M_p|^2 + \frac{d_a}{8}\left(\frac{2}{\delta}+1\right)\int^1_e |w_n|^2 dx + \frac{d_a}{8} \int^1_e |w_{nx}|^2 dx \quad \mbox{a.e. on } [0,T].
    \end{align*}
 Here, by putting $d_* = \min \{d_l, d_a \}$ we obtain
 \begin{align*}
&\frac{1}{2} \frac{d}{dt} \int^1_0 |w_n|^2dx +\frac{7 d_*}{8} \int^1_0 |w_{nx}|^2 dx\\
 &\leq \frac{2}{d_a} |b_1 M_p + b_2 M_p|^2 + \frac{d_a}{8}\left(\frac{2}{\delta}+1\right)\int^1_e |w_n|^2 dx  \quad \mbox{a.e. on } [0,T].
  \end{align*}
  By applying Gronwall's inequality and Lemma \ref{lem1}, we conclude that Lemma \ref{lem4} holds.
     \end{proof}
\begin{lemma}\label{lem-fu}
Let $(u_n, w_n, e_n)$ be a solution of P$(u_{0n}, w_{0n})$ for $n \geq 1$. Then there exists a positive constant $C_4$ such that
\begin{align*}
\left| \int^{T_0}_0 \int^1_0 w_{nt} \eta dx dt \right| \leq C_4|\eta|_{L^2(0,T;W^{1,2}(0,1))} \quad \mbox{for } \eta \in L^2(0,T;W^{1,2}(0,1) \mbox{ and } n \geq 1.
\end{align*}
Namely, $\{ w_{nt} \}$ is bounded in $L^2(0,T;W^{1,2}(0,1)^*)$, where $W^{1,2}(0,1)^*$ is a dual space of $W^{1,2}(0,1)$.
\end{lemma}
\begin{proof}
By applying H\"{o}lder's inequality and  Lemma \ref{lem4}, we have
\begin{align*}
&\left| \int^{T_0}_0 \int^1_0 w_{nt} \eta dx dt \right| \\
&= \left| d_l \int^{T_0}_0 \int^{e_n}_0 w_{nxx} \eta dx dt + d_a \int^{T_0}_0 \int^1_{e_n} w_{nxx} \eta dx dt \right| \\
&\leq d_l \int^{T_0}_0 \int^{1}_0 |w_{nx} \eta_x| dx dt + d_a \int^{T_0}_0 \int^1_{0} |w_{nx} \eta_x| dx dt\\
&\quad \, + \int^{T_0}_0 |b_1 p(u_n(\cdot,1) + \theta_c) - b_2 p(u_b)| |\eta(\cdot,1)|dt\\
&\leq C_5 |\eta|_{L^2(0,T;W^{1,2}(0,1))} + C_6 \int^{T_0}_0 |\eta(\cdot,1)|dt\\
&\leq C_7 |\eta|_{L^2(0,T;W^{1,2}(0,1))} \quad \mbox{for } \eta \in L^2(0,T;W^{1,2}(0,1)) \mbox{ and } n \geq 1,
\end{align*}
where $C_5$, $C_6$ and $C_7$ are positive constants. Here, we note that Lemma \ref{lem-c} implies boundedness of $\{ u_n \}$ in $L^{\infty}(Q(T_0))$. Thus, Lemma \ref{lem-fu} is proved.
\end{proof}
\begin{lemma}\label{esen}
Under the same assumption as in Lemma \ref{lem-fu}, there exists a positive constant $C_8$ such that
\begin{align*}
\int^{T_0}_0 |e'_n|^4 dt \leq C_8 \quad \mbox{for } n \geq 1.
\end{align*}
\end{lemma}
\begin{proof}
Easily, we have
\begin{align*}
&\int^t_0 |e'_n|^4 d\tau \leq \frac{16}{l^4 \delta^4} \int^t_0 (k_l^4 |u_{nx}(\tau, e_n(\tau)-)|^4 + k_a^4 |u_{nx}(\tau, e_n(\tau)+)|^4) d\tau\\
&\hspace{80mm} \quad \mbox{for } n \geq1, 0 \leq t \leq T_n,
\end{align*}
where $\delta$ is the positive constant given in Proposition \ref{prop1}. Here, we note that
\begin{align*}
&\int^t_0  |u_{nx}(\tau, e_n(\tau)+)|^4 d\tau\\
&\leq \int^t_0 \left( \int^1_{e_n(\tau)} \frac{\partial}{\partial x} \left \{ (1-x) |u_{nx}(\tau,x)|^2 \right \} \frac{1}{1-e_n(\tau)} dx \right)^2 d\tau\\
&\leq \frac{1}{\delta^2} \int^t_0 \left( \int^1_{e_n(\tau)} |u_{nx}(\tau,x)|^2  dx + 2 \int^1_{e_n(\tau)} |u_{nxx}(\tau,x)| |u_{nx}(\tau,x)| dx  \right)^2 d\tau\\
&\leq \frac{4}{\delta^2} \int^t_0 \left(  |u_{nx}(\tau)|^4_{L^2(0,1)}   +  |u_{nxx}(\tau)|^2_{L^2(0,1)} |u_{nx}(\tau)|^2_{L^2(0,1)}  \right)^2 d\tau \quad \mbox{for } n \geq1.
\end{align*}
For $\int^t_0  |u_{nx}(\tau, e_n(\tau)-)|^4 d\tau$ we have same estimate. Thus, on account of Lemma \ref{lem-c} this lemma has been proved.
\end{proof}
By Lemmas \ref{lem-c} and \ref{lem4}, the next lemma holds. 
\begin{lemma}\label{convergence}
Let $(u_n, w_n, e_n)$ be a solution of P$(u_{0n}, w_{0n})$ for $n \geq 1$. Then there exists a subsequence $\{ n_j \} \subset \{ n \}$, and $u \in W^{1,2}(0,T_0; L^2(0,1)) \cap L^{\infty}(0,T_0; W^{1,2}(0,1))$, $w \in L^2(0,T_0;W^{1,2}(0,1)) \cap L^{\infty}(0,T_0; L^2(0,1))$ and $e \in W^{1,4}(0,T_0)$ such that
\begin{align}
&u_{n_j} \to u \mbox{ weakly in } W^{1,2}(0, T_0; L^2(0, 1)), \mbox{ weakly$\ast$ in } L^{\infty}(0, T_0; W^{1,2}(0, 1)),\label{con-u1}\\
&\hspace{16mm} \mbox{and in } C([0,T_0] \times [0,1]),\label{con-u}\\
&w_{n_j} \to w \mbox{ weakly in } L^{2}(0, T_0; W^{1,2}(0, 1)), \mbox{ weakly$\ast$ in } L^{\infty}(0, T_0; L^{2}(0, 1)),\label{con-w1}\\
&\hspace{16mm} \mbox{and strongly in } L^{2}(0, T_0; L^2(0, 1)),\label{con-w2}\\
& e_{n_j} \to e \mbox{ in } C([0,T_0]), \mbox{ weakly in } W^{1,4}(0,T_0) \quad \mbox{as } j \to \infty.\label{con-e}
\end{align}
Moreover, it holds that
\begin{align}
w_{n_j}(\cdot, e_{n_j}) \to w(\cdot,e)  \mbox{ in } L^2(0, T_0) \quad \mbox{as } j \to \infty.\label{con-wn}
\end{align}
\end{lemma}
\begin{proof}
It is easy to take a subsequence satisfying \eqref{con-u}, \eqref{con-w1} and \eqref{con-e}. In addition, by  Aubin's compact theorem implies existence of a subsequence such that \eqref{con-w2} holds. To prove  \eqref{con-wn}, we see that
\begin{align*}
&\int^{T_0}_0 |w_{n_j}(t, e_{n_j}(t)) - w(t,e(t))|^2 dt\\
& \leq 2 \int^{T_0}_0 |w_{n_j}(t, e_{n_j}(t)) - w_{n_j}(t,e(t))|^2 dt + 2 \int^{T_0}_0 |w_{n_j}(t, e(t)) - w(t,e(t))|^2 dt\\
&=: I_{1_j}(t) + I_{2_j}(t) \quad \mbox{for a.e. } t \in [0,T_0].
\end{align*}
Thanks to \eqref{wwx} and \eqref{con-e}, we see that
\begin{align*}
I_{1_j}(t) 
&\leq 2 \sup_{0 \leq t \leq T_0}|e_{n_j}(t) - e(t)| \int^{T_0}_0 \int^1_0 |w_{n_jx}(t, x)|^2 dx  dt \\
&\to 0  \quad \mbox{as } j \to \infty.
\end{align*}
Before proving the convergence of $I_{2_j}$, we note that $\delta \leq e_n \leq 1-\delta$ on $[0,T_0]$ for $n \geq 1$. Accordingly, by \eqref{con-w2} and H\"{o}lder's inequality, 
\begin{align*}
&I_{2_j}(t)\\ 
&\leq 2 \int^{T_0}_0\!\!\! \int^{e(t)}_0 \frac{\partial}{\partial x} \left| \frac{x}{e(t)} (w_{n_j}(t,x) - w(t,x)) \right|^2 dx dt\\
&\leq \frac{4}{\delta} \left( \int^{T_0}_0\!\!\! \int^{1}_0 |w_{n_j}(t,x) - w(t,x)|^2 dx dt \right)^{1/2}\\
&\quad \,  + 4 \left( \int^{T_0}_0\!\!\! \int^{1}_0 |w_{n_j}(t,x) - w(t,x)|^2 dx dt \right)^{1/2}\left( \int^{T_0}_0 \!\!\! \int^{1}_0 |w_{n_jx}(t,x) - w_x(t,x)|^2 dx dt \right)^{1/2}\\
&\to 0  \quad \mbox{as } j \to \infty.
\end{align*}
Thanks to these convergences, we have proved this lemma.
\end{proof}
\begin{lemma}\label{u-eq}
Let $u$ be the limit function given by Lemma \ref{convergence}. Then it holds that
\begin{align*}
c_l u_{t} = k_l u_{xx} \quad \mbox{a.e. in } Q_l(T_0,e),\quad c_a u_{t} = k_a u_{xx} \quad \mbox{a.e. in } Q_a(T_0,e).
\end{align*}
\end{lemma}
\begin{proof}
Let $(u_{n_j}, w_{n_j}, e_{n_j})$ be the solution of P$(u_{0n_j}, w_{0n_j})$ for $j \geq 1$. For any $\eta \in C_0^{\infty}(Q_l(T_0,e))$ put $\eta_{n_j}(t,x) = \eta\left(t,\frac{e(t)}{e_{n_j}(t)}\right)$ for $(t,x) \in (0,T_0) \times (0,1)$ and $j \geq 1$. Then we get
 \begin{align}
\int^{T_0}_0 \int^{e_{n_j}}_0 c_l u_{n_jt} \eta_{n_j} dx dt &=\int^{T_0}_0 \int^{e_{n_j}}_0 k_l u_{n_jxx} \eta_{n_j} dxdt\notag\\
&=- \int^{T_0}_0 \int^{e_{n_j}}_0 k_l u_{n_jx} \eta_{n_jx} dx dt\label{ulimit}.
\end{align}
Let $\eta = 0$ on $\overline{Q_l(T_0,e)}$. It is clear that $\eta_{n_j} = 0$ on $\overline{Q_l(T_0,e_{n_j})}$ for $j \geq 1$. Accordingly, for the left hand side of \eqref{ulimit}, we have
 \begin{align*}
&\int^{T_0}_0 \int^{e_{n_j}}_0 c_l u_{n_jt} \eta_{n_j} dx dt  - \int^{T_0}_0 \int^{e}_0 c_l u_t \eta dx dt\\
&=\int^{T_0}_0 \int^1_0 c_l u_{n_jt} \eta_{n_j} dx dt  - \int^{T_0}_0 \int^1_0 c_l u_t \eta dx dt\\
&=\int^{T_0}_0 \int^1_0 c_l u_{n_jt} (\eta_{n_j} - \eta) dx dt  + \int^{T_0}_0 \int^1_0 c_l (u_{n_jt} -u_t) \eta dx dt\\
&=: I_{1_j} + I_{2_j} \quad \mbox{for } j \geq 1.
\end{align*}
Since $\eta_{n_j} \to \eta$ in $L^{2}(0, T_0; L^{2}(0, 1))$ as $j \to \infty$, based on Lemma \ref{lem-c}, we see  that
\begin{align*}
|I_{1_j}| &\leq c_l \left( \int^{T_0}_0 \int^1_0 |u_{n_jt}|^2 dx dt \right)^{1/2} \left( \int^{T_0}_0 \int^1_0 |\eta_{n_j} - \eta|^2 dx dt \right)^{1/2}\\
& \to 0  \quad \mbox{as } j \to \infty.
\end{align*}
Also, by Lemma \ref{convergence} it is easy to see that $I_{2_j} \to 0$ as $j \to \infty$. Thus, we have
 \begin{align*}
\int^{T_0}_0 \int^{e_{n_j}}_0 c_l u_{n_jt} \eta_{n_j} dx dt  \to \int^{T_0}_0 \int^{e}_0 c_l u_t \eta dx dt \quad \mbox{as } j \to \infty.
\end{align*}
Because  $\eta_{n_jx} \to \eta_x$ in $L^{2}(0, T_0; L^{2}(0, 1))$ as $j \to \infty$, similarly to above, we see that
 \begin{align*}
&\int^{T_0}_0 \int^{e_{nj}}_0 k_l u_{njx} \eta_{njx} dx dt\\
 &\to - \int^{T_0}_0 \int^{e}_0 k_l u_{x} \eta_{x} dx dt
= \int^{T_0}_0 \int^{e}_0 k_l u_{xx} \eta dx dt  \quad \mbox{as } j \to \infty.
\end{align*}
Therefore, we have
 \begin{align*}
\int^{T_0}_0 \int^{e}_0 c_l u_t \eta dx dt   = \int^{T_0}_0 \int^{e}_0 k_l u_{xx} \eta dx dt \quad \mbox{for any } \eta \in C_0^{\infty}(Q_l(T_0,e)).
 \end{align*}
 Also, we can get
  \begin{align*}
\int^{T_0}_0 \int_{e}^1 c_a u_t \eta dx dt   = \int^{T_0}_0 \int^{1}_e k_a u_{xx} \eta dx dt \quad \mbox{for any } \eta \in C_0^{\infty}(Q_a(T_0,e)).
 \end{align*}
 Thus, Lemma \ref{u-eq} has been proved.
\end{proof}
\begin{lemma}\label{fbc-con11}
 Let $(u_n, w_n, e_n)$ be a solution of P$(u_{0n}, w_{0n})$ on $[0,T_0]$ for $n \geq 1$ and $\{ n_j \}$ be the subsequence given in Lemma \ref{convergence}. Then it holds that
\begin{align}
&u_{n_jx}(\cdot,0+) \to u_x(\cdot,0+) \quad \mbox{weakly in } L^2(0, T_0),\label{unx0}\\
&u_{n_jx}(\cdot,1-) \to u_x(\cdot,1-) \quad \mbox{weakly in } L^2(0, T_0) \quad \mbox{as } j \to \infty. \label{unx1}
\end{align}
\end{lemma}
\begin{proof}
Lemmas \ref{convergence} and \ref{u-eq} imply that $u_{n_{j}xx} \to u_{xx}$ weakly in $L^2(0,T_0; L^2(0, \delta))$ as $j \to \infty$, where $\delta$ is a positive constant satisfying $e_{n_j} \geq \delta$ on $[0,T_0]$ for $j \geq 1$. Therefore, $u_{n_{j}x} \to u_{x}$ weakly in $L^2(0,T_0; W^{1,2}(0, \delta))$ as $j \to \infty$. Here, for any $v \in L^2(0, T_0)$ let  $F_v(\eta) = \int^{T_0}_0 v(t) \eta(t,0) dt \mbox{ for } \eta \in L^2(0,T_0; W^{1,2}(0, \delta))$. Since the functional $F_v$ is linear and bounded on $L^2(0,T_0; W^{1,2}(0, \delta))$, $u_{n_jx}(\cdot,0+) \to u_x(\cdot,0+)$ weakly in  $L^2(0, T_0)$ as $j \to \infty$. Similarly, it holds that $u_{n_jx}(\cdot,1-) \to u_x(\cdot,1-)$ weakly in  $L^2(0, T_0)$ as $j \to \infty$. Thus, this lemma is true.
\end{proof}
Based on \cite{Ai}, we can show the convergence of the free boundary. To do this, we provide the following lemma.
\begin{lemma}\label{hate}
Let $(u_n, w_n, e_n)$ be a solution of P$(u_{0n}, w_{0n})$ for $n \geq 1$. Then for any $\varepsilon > 0$, there exist $\hat{e} \in C^2([0,T])$ and $n_{\varepsilon}$ such that $0 \leq e_n - \hat{e} \leq 4 \varepsilon$ for $n \geq n_{\varepsilon}$.
\end{lemma}
\begin{proof}
Since $e_n \to e$ in  $C([0,T_0])$ as $n \to \infty$, there exists $n_0 \in \mathbb{N}$ such that $|e_n - e|_{C([0,T_0])} < \varepsilon$ for $n \geq n_0$. By putting $\overline{e} = e - 2 \varepsilon$, we have $\overline{e} + \varepsilon < e_n < \overline{e} + 3 \varepsilon$ on $[0,T_0]$ for $n \geq n_0$. Moreover, by putting $\hat{e} = J_{\varepsilon} \ast \overline{e}$, we have $|\hat{e} - \overline{e}| < \varepsilon$ on $[0,T_0]$, where $J_{\varepsilon}$ is a standard mollifier in $\mathbb{R}$ and $\ast$ is the convolution. Thus, we have
 \begin{align*}
 \hat{e} &\leq |\hat{e} - \overline{e}| + \overline{e}\\
 &< e_n\\
 &\leq 3 \varepsilon + \hat{e} + |\hat{e} - \overline{e}|\\
 &\leq 4 \varepsilon + \hat{e} \quad \mbox{on } [0,T_0] \mbox{ for } n \geq n_0.
 \end{align*}
 Hence, it holds that $\hat{e} < e_n < 4 \varepsilon + \hat{e}$ on $[0,T_0]$ for $n \geq n_0$, namely, Lemma \ref{hate} is true.
\end{proof}
\begin{lemma}\label{fbc-con1}
(cf.  \cite{Ai} Lemma 2.4) Let $(u_n, w_n, e_n)$ be a solution of P$(u_{0n}, w_{0n})$ on $[0,T_0]$ for $n \geq 1$ and $\{ n_j \}$ be the subsequence given in Lemma \ref{convergence}. Then it holds that
\begin{align}
&u_{n_{j}x} \to u_x \quad \mbox{in } L^2(Q(T_0)),\label{unx}\\
&u_{n_jx}(\cdot, e_{n_j}\pm) \to u_x(\cdot, e\pm) \quad \mbox{in } L^2(0, T_0),\label{unxe}\\
&w_{n_j}(\cdot, e_{n_j})e'_{n_j} \to w(\cdot, e)e' \quad \mbox{weakly in } L^2(0, T_0) \quad \mbox{as } j \to \infty.\label{fbc-z}
\end{align}
\end{lemma}
\begin{proof} First, by \eqref{con-u} and \eqref{con-e} we have $u(t,e(t)) = 0$ for $0 \leq t \leq T_0$, and thanks to Lemma \ref{fbc-con11}, \eqref{BCl} and \eqref{BCau} hold. Based on these facts, we shall prove \eqref{unx}. Since $u_{n_{j}x} \to u_x$ weakly in $L^2(0,T_0; L^2(0,1))$ as $j \to \infty$, it is sufficient to show $|u_{n_{j}x}|_{L^2(0,T_0; L^2(0,1))} \to |u_x|_{L^2(0,T_0; L^2(0,1))}$ as $j \to \infty$.  It is clear that $\liminf_{j \to \infty} |u_{n_{j}x}|_{L^2(0,T_0; L^2(0,1))}^2 \geq |u_{x}|^2_{L^2(0,T_0;}$\\$_{L^2(0,1))}$. Accordingly, in order to show $\limsup_{j \to \infty} |u_{n_{j}x}|_{L^2(0,T_0; L^2(0,1))}^2 \leq |u_{x}|_{L^2(0,T_0;L^2(0,1))}^2$, for $\varepsilon > 0$ let $\hat{e}_{\varepsilon}$ and $n_{\varepsilon}$ be the function and the constant given by Lemma \ref{hate}. Also, take a positive integer $j_{\varepsilon}$ satisfying  $n_j \geq n_{\varepsilon}$ for $j \geq j_{\varepsilon}$ By integration by parts and Lemma \ref{u-eq}, we see that
 \begin{align*}
& |u_{n_{j}x}|_{L^2(0,T_0; L^2(0,1))}^2 -  |u_x|_{L^2(0,T_0; L^2(0,1))}^2\\
 &= \int^{T_0}_0 \int^{e}_0 \frac{c_l}{k_l} u_t u dx dt  -\int^{T_0}_0 \int^{e_{n_j}}_0 \frac{c_l}{k_l} u_{n_jt} u_{n_j} dx dt \\
& \quad \, + \int^{T_0}_0 \int_{e}^1 \frac{c_a}{k_a} u_t u dx dt  -\int^{T_0}_0 \int_{e_{n_j}}^1 \frac{c_a}{k_a} u_{n_jt} u_{n_j} dx dt \\
 & \quad \,  + \int^{T_0}_0 u_{x}(\cdot,1) u(\cdot,1) dt  -\int^{T_0}_0 u_{n_jx}(\cdot,1) u_{n_j}(\cdot,1) dt\\
 &=: I_{1_j} + I_{2_j} + I_{3_j} \quad \mbox{for } j \geq j_{\varepsilon}.
  \end{align*}
  In order to handle $I_{1_j}$ we note that $0 \leq e_{n_j} - \hat{e}_{\varepsilon} \leq 4 \varepsilon$ on $[0,T_0]$ for $j \geq j_{\varepsilon}$, namely, $0 \leq e - \hat{e}_{\varepsilon} \leq 4 \varepsilon$ on $[0,T_0]$. Obviously, it holds that
   \begin{align*}
| I_{1_j}|&\leq \left| \int^{T_0}_0 \int^{\hat{e}_{\varepsilon} }_0 \frac{c_l}{k_l} u_t u dx dt + \int^{T_0}_0 \int^{e}_{\hat{e}_{\varepsilon} } \frac{c_l}{k_l} u_t u dx dt \right.\\
&\quad \, \left. -\int^{T_0}_0 \int^{\hat{e}_{\varepsilon}}_0 \frac{c_l}{k_l} u_{n_jt} u_{n_j} dx dt  -\int^{T_0}_0 \int^{e_{n_j}}_{\hat{e}_{\varepsilon}} \frac{c_l}{k_l} u_{n_jt} u_{n_j} dx dt \right|\\
&\leq \left|\int^{T_0}_0 \int^{\hat{e}_{\varepsilon}}_0 \frac{c_l}{k_l} (u_t - u_{n_jt} ) u dx dt \right| + \left|\int^{T_0}_0 \int^{\hat{e}_{\varepsilon}}_0 \frac{c_l}{k_l} u_{n_jt} (u - u_{n_j} ) dx dt \right|\\
&\quad \, + \left|\ \int^{T_0}_0 \int^{e}_{\hat{e}_{\varepsilon} } \frac{c_l}{k_l} u_t u dx dt \right| + \left| \int^{T_0}_0 \int^{e_{n_j}}_{\hat{e}_{\varepsilon}} \frac{c_l}{k_l} u_{n_jt} u_{n_j} dx dt \right|\\
&=:  I_{11_j} + I_{12_j} + I_{13_j}  + I_{14_j} \quad \mbox{for } j \geq j_{\varepsilon}.
  \end{align*} 
 Due to the weakly convergence of $\{ u_{n_jt}\}$, we see that $I_{11_j} \to 0$ as $j \to \infty$.  In addition,
    thanks to Lemma \ref{lem-c} and \eqref{con-u}, we have
  \begin{align*}  
  I_{12_j} &\leq  \frac{c_l}{k_l} T_0^{1/2} |u - u_{n_j}|_{ L^{\infty}(Q(T_0))}  \left( \int^{T_0}_0 \int^{1}_0  |u_{n_jt}|^2 dx dt \right)^{1/2} \\
  &\to 0 \quad \mbox{as } j \to \infty.
   \end{align*}  
   By Lemma \ref{lem-c}, we obtain
     \begin{align*}  
  I_{13_j} &\leq \frac{c_l}{k_l}  |u |_{ L^{\infty}(Q(T_0))}  \int^{T_0}_0 \int^{e}_{\hat{e}_{\varepsilon}}  |u_t| dx dt\\
  &\leq 2 \varepsilon^{1/2} T_0^{1/2} \frac{c_l}{k_l}  |u |_{ L^{\infty}(Q(T_0))} \left(  \int^{T_0}_0 \int^1_0|u_t|^2 dx dt \right)^{1/2} \quad \mbox{for } j \geq j_{\varepsilon}.
    \end{align*}   
  Because of \eqref{con-u} and \eqref{con-e}, it holds that
    \begin{align*}  
  I_{14_j} &\leq  \frac{c_l}{k_l}|u_{n_j} |_{ L^{\infty}(Q(T_0))}  \int^{T_0}_0 \int^{e_{n_j}}_{\hat{e}_{\varepsilon}}  |u_{n_jt}| dx dt\\
   &\leq  \frac{c_l}{k_l}  T_0^{1/2}  |u_{n_j}|_{ L^{\infty}(Q(T_0))}  \sup_{0 \leq t \leq T_0}|e_{n_j}(t) - \hat{e}_{\varepsilon}(t)|^{1/2} \left( \int^{T_0}_0 \int^1_0 |u_{n_jt}|^2 dx dt\right)^{1/2}\\
   &\leq C_9 \varepsilon^{1/2} \quad \mbox{for } j \geq j_{\varepsilon},
   \end{align*}  
   where $C_9$ is a positive constant independent of $j$ and $\varepsilon > 0$. Thus, $\limsup_{j \to \infty}|I_{1_j}| \leq C_{10} \sqrt{\varepsilon}$ for $\varepsilon$,  where $C_{10}$ is a positive constant. The similar estimate holds for $I_{2_j}$.
\par Next, we estimate $I_{3_j}$. Easily, we have
\begin{align*}
I_{3_j} &\leq \left| \int^{T_0}_0 (u_{n_jx}(\cdot,1)   -  u_{x}(\cdot,1)) u(\cdot,1)dt \right| +  \left| \int^{T_0}_0 u_{n_jx}(\cdot,1) (u_{n_j}(\cdot,1) -  u(\cdot,1)) dt \right|\\
&=:  I_{31_j} + I_{32_j} \quad \mbox{for } j \geq j_{\varepsilon}.
\end{align*}
Since  $u \in L^{\infty}(Q(T_0))$ and  \eqref{unx1}, we get $I_{31_j} \to 0$ as $j \to \infty$. Thanks to Lemma \ref{lem-c}, \eqref{con-u} and the Gagliardo-Nirenberg inequality, we have
\begin{align*}
I_{32_j} &\leq  T_0^{1/2}  |u_{n_j}-  u|_{ L^{\infty}(Q(T_0))}  \left( \int^{T_0}_0 \left( \frac{2}{\delta^2}|u_{n_jx}|^2_{L^2(1-\delta,1)} +  |u_{n_jxx}|^2_{L^2(1-\delta,1)} \right) dt \right)^{1/2}\\
&\to 0 \quad \mbox{as } j \to \infty.
\end{align*}
 Hence, \eqref{unx} has been proved. 
\par Next, we show  \eqref{unxe}. Easily, we see that
\begin{align*}
&\int^{T_0}_0 | u_{n_jx}(\cdot, e_{n_j}-) - u_x(\cdot, e-) |^2 dt\\
&\leq 9 \int^{T_0}_0 | u_{n_jx}(\cdot, e_{n_j}-) - u_{n_jx}(\cdot, \hat{e}_{\varepsilon}-) |^2 dt  + 9 \int^{T_0}_0 | u_{n_jx}(\cdot, \hat{e}_{\varepsilon}-) - u_{x}(\cdot, \hat{e}_{\varepsilon}-) |^2 dt\\
&\quad \, + 9 \int^{T_0}_0 | u_{x}(\cdot, \hat{e}_{\varepsilon}-) - u_{x}(\cdot, e-) |^2 dt\\
&=: I_{4_j} + I_{5_j} + I_{6_j} \quad \mbox{for } j \geq j_{\varepsilon}.
\end{align*}
Thanks to Lemmas \ref{lem-c} and  \ref{hate}, we see that
\begin{align*}
I_{4_j} &\leq 9 \int^{T_0}_0 |e_{n_j} - \hat{e}_{\varepsilon}| \int^{e_{n_j}}_{\hat{e}_{\varepsilon}} | u_{n_jxx}|^2dx dt\\
&\leq 36 \frac{c_l^2}{k_l^2} \varepsilon \int^{T_0}_0 \int^1_0 |u_{n_jt}|^2 dx dt\\
&\leq R_1 \varepsilon \quad \mbox{for } j \geq j_{\varepsilon},
\end{align*}
where $R_1$ is a positive constant independent of $\varepsilon$ and $j$. Similarly, it holds that $I_{6_j} \leq R_2 \varepsilon$ for some positive constant. Moreover, we have
\begin{align*}
I_{5_j} &= 9 \int^{T_0}_0 \int^{\hat{e}_{\varepsilon}}_0 \left|  \frac{\partial}{\partial y} \frac{y}{\hat{e}_{\varepsilon}} (u_{n_jx}(\cdot, y) - u_{x}(\cdot, y))\right|^2 dy dt\\
&\leq \frac{18}{\delta^2} \left( \int^{T_0}_0 \int^1_0 |u_{n{_jx}} - u_x|^2 dxdt \right)^{1/2}\\
&\quad \, + 18 \left( \int^{T_0}_0 \int^1_0 |u_{n{_jx}} - u_x|^2 dxdt \right)^{1/2} \left( \int^{T_0}_0 \int^1_0 |u_{n_jt} - u_t|^2 dxdt \right)^{1/2} \quad  \mbox{for } j \geq j_{\varepsilon}.
\end{align*}
By Lemmas \ref{lem-c} and \ref{convergence}, \eqref{unx} implies $I_{5_j} \to 0$ as $j \to \infty$. From the above convergences, it holds that $u_{n_jx}(\cdot, e_{n_j}-) \to u_x(\cdot, e-)$ in $L^2(0, T_0)$ as $j \to \infty$.  Similarly, we can show $u_{n_jx}(\cdot, e_{n_j}+) \to u_x(\cdot, e+)$ in $L^2(0, T_0)$ as $j \to \infty$. Thus, \eqref{unxe} has been proved.
 \par Finally, we show \eqref{fbc-z}. Let $\eta \in L^2(0, T_0)$. There exists a sequence $\{ \eta_k\} \subset C^{\infty}([0,T_0])$ such that $\eta_k \to \eta$ in $L^2(0, T_0)$ as $k \to \infty$. Immediately, we have
\begin{align*}
&\left | \int^{T_0}_0( w_{n_j}(\cdot, e_{n_j})e'_{n_j} - w(\cdot,e)e') \eta dt \right |\\
&\leq \left| \int^{T_0}_0( w_{n_j}(\cdot, e_{n_j})e'_{n_j} - w(\cdot,e)e') (\eta - \eta_k) dt \right | +\left | \int^{T_0}_0( w_{n_j}(\cdot, e_{n_j})e'_{n_j} - w(\cdot,e)e'_{n_j}) \eta_k dt \right |\\
&\quad \, +\left| \int^{T_0}_0 w(\cdot, e)(e'_{n_j} - e') \eta_k dt \right |\\
&=: I_{7_j} + I_{8_j} + I_{9_j} \quad  \mbox{for } j \geq 1.
\end{align*}
In order to prove a convergence of $\{ I_{7_j} \}$, we give some estimates for $\{ w_{n_j}(\cdot, e_{n_j}) \}$ in $L^4(0, T_0)$. First, by the  Gagliardo-Nirenberg inequality, we see that
\begin{align*}
&\int^{T_0}_0 |w_{n_j}(t, e_{n_j}(t)|^4 dt\\
&\leq C \int^{T_0}_0 \left( |w_{n_j}(t)|^4_{L^2(0,1)} + |w_{n_j}(t)|^2_{L^2(0,1)}|w_{n_jx}(t)|^2_{L^2(0,1)} \right) dt \quad  \mbox{for } j \geq 1,
\end{align*}
where $C$ is a positive constant. Accordingly,  $\{ w_{n_j}(\cdot, e_{n_j}) \}$ is bounded in $L^4(0, T_0)$. Similarly, we have $w(\cdot,e) \in L^4(0,T_0)$. From these estimates it follows that
\begin{align*}
|I_{7_j}| &\leq \left(|w_{n_j}(\cdot, e_{n_j})|_{L^4(0, T_0)} |e'_{n_j}|_{L^4(0, T_0)} + |w(\cdot, e)|_{L^4(0, T_0)} |e'|_{L^4(0, T_0)} \right) |\eta - \eta_k|_{L^2(0, T_0)}\\
&\hspace{100mm} \mbox{for } j \geq 1.
\end{align*}
Hence, for any $\hat{\varepsilon} > 0$ there exists $k_{\hat{\varepsilon}}$ such that $|I_{7_j}| \leq \hat{\varepsilon}$ for any $j \geq 1$ and $k \geq k_{\hat{\varepsilon}}$. From now on, we fix $k \geq k_{\hat{\varepsilon}}$. Easily, we get
\begin{align*}
|I_{8_j}| &\leq |w_{n_j}(\cdot, e_{n_j}) - w(\cdot, e)|_{L^2(0, T_0)} |e'_{n_j}|_{L^4(0, T_0)} | \eta_k|_{L^4(0, T_0)}\quad \mbox{for } j \geq 1.
\end{align*}
In addition, because of $w(\cdot,e)\eta_k \in L^2(0,T_0)$, $I_{9_j} \to 0$ as $j \to \infty$ for fixed $k$. Thus, we have proved this lemma.
\end{proof}
\begin{proof}[Proof of Theorem \ref{th1}]
 Lemmas \ref{lem-c}, \ref{lem4} and \ref{esen} imply Definition  \ref{def1} (S1), (S2) and (S3). Also, by Lemmas \ref{u-eq}, \ref{fbc-con11} and \ref{fbc-con1}, $u$ satisfies \eqref{EQl}, \eqref{BCl}, \eqref{BCau} and $u(0)=u_0$ on $[0,1]$.  Thanks to \eqref{unxe} and \eqref{fbc-z},  \eqref{FBP1} holds.
 \par  Finally, we shall show \eqref{weak}. Let $\eta \in W^{1,2}(0, T_0; L^2(0,1)) \cap L^{2}(0, T_0; W^{1,2}(0,1))$ and $\eta(T_0)=0$. By multiplying $\eta$ on both sides of  \eqref{EQa} and integrating it with $x$ on $[0, e_{n_j}(t)]$, $[e_{n_j}(t), 1]$, respectively, we have
\begin{align*}
&\int^{T_0}_0 \int^1_0 w_{n_jt} \eta dx dt = \int^{T_0}_0 \int^{e_{n_j}}_0 d_l  w_{n_jxx} \eta dx dt + \int^{T_0}_0 \int^1_{e_{n_j}} d_a w_{n_jxx} \eta dx dt \quad \mbox{for } j \geq 1.
\end{align*}
For the left hand side, easily, we see that
\begin{align*}
&\int^{T_0}_0 \int^1_0 w_{n_jt} \eta dx dt 
= - \int^1_0 w_{0n_j}\eta(0) dx - \int^{T_0}_0 \int^1_0 w_{n_j} \eta_t dx dt  \quad \mbox{for } j \geq 1.
\end{align*}
Because of \eqref{con-w2},  it is clear that $\int^1_0 w_{0n_j}\eta(0) dx \to \int^1_0 w_{0}\eta(0) dx$ and  $\int^{T_0}_0 \int^1_0 w_{n_j} \eta_t dx dt \to \int^{T_0}_0 \int^1_0 w \eta_t dx dt$ as $j \to \infty$. On the right hand side, by integration by parts, it holds that
\begin{align*}
& \int^{T_0}_0 \int^{e_{n_j}}_0 d_l  w_{n_jxx} \eta dx dt + \int^{T_0}_0 \int^1_{e_{n_j}} d_a w_{n_jxx} \eta dx dt \\
&=-d_l \int^{T_0}_0 \int^{e_{n_j}}_0  w_{n_jx} \eta_x dx dt - d_a \int^{T_0}_0 \int^1_{e_{n_j}} w_{n_jx} \eta_x dx dt + d_a \int^{T_0}_0  w_{n_jx}(\cdot,1) \eta(\cdot,1) dt \\
&=: I_{1_j} + I_{2_j} + I_{3_j}   \quad \mbox{for } j \geq 1.
 \end{align*}
On account of $\eta \in L^{2}(0, T_0; W^{1,2}(0,1))$ there exists $\{ \eta_k \} \subset L^{2}(0, T_0; W^{2,2}(0,1))$ such that $\eta_k \to \eta$ in $L^{2}(0, T_0; W^{1,2}(0,1))$ as $k \to \infty$.  On $I_{1_j}$, we see that
 \begin{align*}
& \left| \int^{T_0}_0 \int^{e_{n_j}}_0  w_{n_jx} \eta_x dx dt - \int^{T_0}_0 \int^e_0  w_x \eta_x dx dt \right|\\
&\leq \left|  \int^{T_0}_0 \int^{e_{n_j}}_0  w_{n_jx} \eta_x dx dt -  \int^{T_0}_0 \int^{e_{n_j}}_0  w_{n_jx} \eta_{kx} dx dt \right|\\
&\quad \, + \left|  \int^{T_0}_0 \int^{e_{n_j}}_0  w_{n_jx} \eta_{kx} dx dt -  \int^{T_0}_0 \int^e_0  w_{n_jx} \eta_{kx} dx dt \right|\\
&\quad \, + \left|  \int^{T_0}_0 \int^e_0  w_{n_jx} \eta_{kx} dx dt -  \int^{T_0}_0 \int^e_0  w_x \eta_{kx} dx dt\right|\\
&\quad \, +\left|  \int^{T_0}_0 \int^e_0  w_x \eta_{kx} dx dt  -  \int^{T_0}_0 \int^e_0  w_x \eta_x dx dt \right|\\
&=: I_{11_j} + I_{12_j} + I_{13_j}  + I_{14_j}  \quad \mbox{for } j \geq 1.
  \end{align*}
  By Lemma \ref{lem4} and H\"{o}lder's i inequality, we have
 \begin{align*}
 I_{11_j} &\leq  \left( \int^{T_0}_0 \int^1_0 |w_{n{_jx}} |^2 dxdt \right)^{1/2} \left( \int^{T_0}_0 \int^1_0 |\eta_x - \eta_{kx}|^2 dxdt \right)^{1/2}\\
 &\to 0 \quad \mbox{as } k \to \infty.
   \end{align*}  
   For $I_{12_j}$, by integration by parts, we see that
    \begin{align*}
 I_{12_j} &\leq  \left|  \int^{T_0}_0 \int^{e_{n_j}}_e  w_{n_jx} \eta_{kx} dx dt \right|\\
 &\leq  \left|  \int^{T_0}_0   w_{n_j}(\cdot,e_{n_j}) \eta_{kx}(\cdot,e_{n_j})  dt  -  \int^{T_0}_0   w_{n_j}(\cdot,e) \eta_{kx}(\cdot,e)  dt \right.\\
 &\left. \quad \, -  \int^{T_0}_0 \int^{e_{n_j}}_e  w_{n_j} \eta_{kxx} dx dt \right|\\
 &\leq \left|  \int^{T_0}_0   (w_{n_j}(\cdot,e_{n_j}) -  w_{n_j}(\cdot,e)) \eta_{kx}(\cdot,e_{n_j})  dt\right|\\
&\quad \,  + \left|  \int^{T_0}_0   w_{n_j}(\cdot,e) (\eta_{kx}(\cdot,e_{n_j}) -\eta_{kx}(\cdot,e))  dt \right| + \left|  \int^{T_0}_0 \int^{e_{n_j}}_e  w_{n_j} \eta_{kxx} dx dt \right|\\
&=: J_{1_j} + J_{2_j} + J_{3_j}   \quad \mbox{for } j \geq 1.
   \end{align*}  
   On $J_{1_j}$, it holds that
      \begin{align}
J_{1_j} &\leq   \left|  \int^{T_0}_0 \left(  \int^{e_{n_j}}_{e} w_{n_j x}(\cdot, x) dx\right) \eta_{kx}(\cdot,e_{n_j})  dt\right|\notag\\
&\leq \int^{T_0}_0  |e_{n_j} - e|^{1/2} | w_{n_j x}|_{L^2(0,1)}|\eta_{kx}(\cdot,e_{n_j})| dt\notag\\\
&\leq  \sup_{0 \leq t \leq T_0}|e_{n_j}(t) - e(t)|^{1/2} |w_{n_j x}|_{L^2(0,T_0;L^2(0,1))} |\eta_{kx}(\cdot,e_{n_j})|_{L^2(0,T_0)} \quad \mbox{for } j \geq 1.\label{conJ1}
     \end{align}
  Thanks to Lemmas \ref{lem4}  and \ref{convergence}, we have $J_{1_j} \to 0$ as $j \to \infty$ for each $k \geq 1$. By Lemmas \ref{lem4}, \ref{convergence} and the Gagliardo-Nirenberg inequality, we obtain
    \begin{align}
J_{2_j} &\leq  \sup_{0 \leq t \leq T_0}|e_{n_j}(t) - e(t)|^{1/2}  |\eta_{kxx}|_{L^2(0,T_0;L^2(0,1))}  \left(  \int^{T_0}_0 |w_{n_j}(\cdot,e)|^2 dt \right)^{1/2}\notag\\\
&\leq C  \sup_{0 \leq t \leq T_0}|e_{n_j}(t) - e(t)|^{1/2}  |\eta_{kxx}|_{L^2(0,T_0;L^2(0,1))}\notag\\\
&\quad \, \times \left( \int^{T_0}_0 \left( |w_{n_j}(t)|^2_{L^2(0,1)} + |w_{n_j}(t)|_{L^2(0,1)}|w_{n_jx}(t)|_{L^2(0,1)} \right) dt \right)^{1/2}\notag\\\
&\to 0 \quad \mbox{as } j \to \infty \mbox{ for each } k \geq 1,\label{conJ2}
     \end{align}     
     where $C$ is a positive constant. Because of Lemma  \ref{convergence} and Sobolev embedding, we see that
         \begin{align}
J_{3_j} &\leq   |\eta_{kxx}|_{L^2(0,T_0;L^2(0,1))}  \left(  \int^{T_0}_0 \int^{e_{n_j}}_e  |w_{n_j}|^2  dx dt \right)^{1/2}\notag\\\
&\leq   |\eta_{kxx}|_{L^2(0,T_0;L^2(0,1))}  \sup_{0 \leq t \leq T_0}|e_{n_j}(t) - e(t)|^{1/2} |w_{n_j}|_{L^2(0,T_0;W^{1,2}(0,1))}\notag\ \\
&\to 0 \quad \mbox{as } j \to \infty \mbox{ for each } k \geq 1.\label{conJ3}
           \end{align}  
Therefore, \eqref{conJ1} - \eqref{conJ3} imply that $ I_{12_j} \to 0$ as $j \to \infty$ for each $k \geq 1$. Due to \eqref{con-w1}, $ I_{13_j} \to 0$ as $j \to \infty$ for each $k \geq 1$. Thanks to $\eta_{kx} \to \eta_x$ in $L^{2}(0, T_0; L^2(0,1))$ as $k \to \infty$, $ I_{14_j} \to 0$ as $k \to \infty$. Hence, these convergences yield that $ I_{1_j} \to -d_l \int^{T_0}_0 \int^{e}_0  w_{x} \eta_x dx dt$ as $j \to \infty$. Similarly, $ I_{2_j} \to -d_a \int^{T_0}_0 \int^{1}_e  w_{x} \eta_x dx dt$ as $j \to \infty$.
\par Finally, we show $I_{3_j} \to  \int^{T_0}_0  (b_1 p(u(\cdot,1) + \theta_c) -b_2p(u_b)) \eta(\cdot,1) dt$ as $j \to \infty$. By \eqref{BCaw}, \eqref{con-u} and the Gagliardo-Nirenberg inequality, we obtain
  \begin{align*}
&\left|  d_a \int^{T_0}_0  w_{n_jx}(\cdot,1) \eta(\cdot,1) dt +  \int^{T_0}_0  (b_1 p(u(\cdot,1) + \theta_c) -b_2p(u_b)) \eta(\cdot,1) dt \right|\\
&\leq b_1 \left|  \int^{T_0}_0  (p(u_{n_j}(\cdot,1) ) - p(u(\cdot,1) ))  \eta(\cdot,1) dt  \right|\\
&\leq b_1 M_p   \int^{T_0}_0  |u_{n_j}(\cdot,1)  - u(\cdot,1) | |\eta(\cdot,1)| dt  \\
&\leq b_1 M_p   \left( \int^{T_0}_0  |u_{n_j}(\cdot,1)  - u(\cdot,1) |^2 dt \right)^{1/2}  \left( \int^{T_0}_0  |\eta(\cdot,1)|^2 dt \right)^{1/2} \\
&\to 0 \quad \mbox{as } j \to \infty.
   \end{align*}   
From the argument as above, it follows \eqref{weak}.  Thus, we can prove Theorem \ref{th1}.

\end{proof}

\section{Conclusion}
In this paper, we have established existence of solutions by introducing a weak formulation to $w$. The proof relies on a standard approximation method for initial data, with the crucial step being the convergence of approximate solutions via uniform estimates derived from the Gagliardo-Nirenberg inequality. In  the future work, we plan to investigate the behavior of the free boundary, in particular, whether it convergences to the fixed boundary in finite time.





\end{document}